\documentclass[10pt]{amsart}
\usepackage{amsfonts}
\usepackage{amsbsy}
\usepackage{tikz}
\usepackage{amsmath}
\usepackage{amssymb}
\usepackage{amsthm}
\usepackage{syntonly}
\usepackage{url} 
\usepackage{amscd} 
\usepackage{tikz-cd} 
\usepackage{bm} 
\usepackage{mathdots} 

\usepackage[colorlinks,linkcolor=brown, anchorcolor=blue, citecolor=green ]{hyperref}
\usepackage{graphicx} 

\usepackage{mathrsfs} 
\usepackage{comment} 
\usepackage{bbm} 

\usepackage[utf8x]{inputenc}
    \usepackage{ucs}
\newcommand\quotient[2]{
        \mathchoice
            {
                \text{\raise1ex\hbox{$#1$}\Big/\lower1ex\hbox{$#2$}}%
            }
            {
                #1\,/\,#2
            }
            {
                #1\,/\,#2
            }
            {
                #1\,/\,#2
            }
    }

\newcommand{\R}{\mathbb{R}} 
\newcommand{\Z}{\mathbb{Z}}
\newcommand{\Q}{\mathbb{Q}}
\newcommand{\C}{\mathbb{C}}

\newcommand{\bma}{\bm{\mathrm{a}}}
\newcommand{\bmA}{\bm{\mathrm{A}}}

\newcommand{\bmD}{\bm{\mathrm{D}}}

\newcommand{\bmF}{\bm{\mathrm{F}}}
\newcommand{\bmG}{\bm{\mathrm{G}}}
\newcommand{\bmH}{\bm{\mathrm{H}}}
\newcommand{\bmL}{\bm{\mathrm{L}}}
\newcommand{\bmM}{\bm{\mathrm{M}}}

\newcommand{\bmP}{\bm{\mathrm{P}}}
\newcommand{\bmQ}{\bm{\mathrm{Q}}}
\newcommand{\bmS}{\bm{\mathrm{S}}}

\newcommand{\bmU}{\bm{\mathrm{U}}}

\newcommand{\bmV}{\bm{\mathrm{V}}}

\newcommand{\bmX}{\bm{\mathrm{X}}}

\newcommand{\bmZ}{\bm{Z}}

\newcommand{\rmA}{\mathrm{A}}

\newcommand{\rmF}{\mathrm{F}}
\newcommand{\rmG}{\mathrm{G}}
\newcommand{\rmH}{\mathrm{H}}
\newcommand{\rmK}{\mathrm{K}}
\newcommand{\rmL}{\mathrm{L}}
\newcommand{\rmM}{\mathrm{M}}
\newcommand{\rmP}{\mathrm{P}}
\newcommand{\rmQ}{\mathrm{Q}}

\newcommand{\rmS}{\mathrm{S}}
\newcommand{\rmU}{\mathrm{U}}
\newcommand{\rmX}{\mathrm{X}}
\newcommand{\rmY}{\mathrm{Y}}

\newcommand{\frakS}{\frak{S}}

\newcommand{\scrA}{\mathscr{A}}

\newcommand{\scrF}{\mathscr{F}}

\newcommand{\calO}{\mathcal{O}}

\newcommand{\wtmu}{\widetilde{\mu}}

\newcommand{\ep}{\varepsilon}

\newtheorem{thm}{Theorem}[section]

\newtheorem{coro}[thm]{Corollary}

\newtheorem{lem}[thm]{Lemma}
\newtheorem{prop}[thm]{Proposition}

\newcommand{\bs}{\backslash}  
\newcommand{\la}{\langle}
\newcommand{\ra}{\rangle}
\newcommand{\norm}[1]{\left\lVert#1\right\rVert}

\DeclareMathOperator{\SL}{SL}

\DeclareMathOperator{\SO}{SO}

\DeclareMathOperator{\Ad}{Ad}

\DeclareMathOperator{\BS}{BS}

\DeclareMathOperator{\Lie}{Lie}
\DeclareMathOperator{\Leb}{Leb}

\DeclareMathOperator{\rank}{rank}

\DeclareMathOperator{\Split}{split}

\begin{document}

\title[Equidistribution on the boundary]{Equidistribution of translates of a homogeneous measure on the Borel--Serre boundary}
\author[R.Zhang]{Runlin Zhang}

\email{zhangrunlinmath@outlook.com}

\begin{abstract}
Let $\bmG$ be a semisimple linear algebraic group defined over rational numbers, $\rmK$ be a maximal compact subgroup of its real points and $\Gamma$ be an arithmetic lattice. One can associate a probability measure $\mu_{\rmH}$ on  $\Gamma \bs \rmG$ for each subgroup $\bmH$ of $\bmG$ defined over $\Q$ with no non-trivial rational characters. As G acts on  $\Gamma \bs \rmG$ from the right, we can push forward this measure by elements from $\rmG$. 
By pushing down these measures to $\Gamma \bs \rmG/\rmK$,
we call them homogeneous. It is a natural question to ask what are the possible weak-$*$ limits of homogeneous measures. In the non-divergent case this has been answered by Eskin--Mozes--Shah. In the divergent case Daw--Gorodnik--Ullmo prove a refined version in some non-trivial compactifications of  $\Gamma \bs \rmG/\rmK$ for $\bmH$ generated by real unipotents. In the present article we build on their work and generalize the theorem to the case of general $\bmH$ with no non-trivial rational characters. Our results rely on (1) a non-divergent criterion on $\SL_n$ proved by geometry of numbers and a theorem of Kleinbock--Margulis;  (2) relations between partial Borel--Serre compactifications associated with different groups proved by geometric invariant theory and reduction theory.
\end{abstract}

\newpage

\maketitle

\tableofcontents

\section{Introduction}

Let $\bmG$ be a connected semisimple algebraic group over $\Q$ and $\Gamma\subset \bmG(\Q)$ be an arithmetic lattice. Let $\bmH \leq \bmG$ be a $\Q$-subgroup with no non-trivial $\Q$-characters. Then there exists a unique $\rmH$-invariant ($\rmH:=\bmH(\R)^{\circ}$) probability measure $\mu_{\rmH}$ supported on $\Gamma\bs \Gamma\rmH$. 

Now, suppose that we are given a sequence $(g_n)$ from $\rmG$ (:=$\bmG(\R)^{\circ}$) and a sequence $(\bmH_n)$ of $\Q$-subgroups of $\bmG$ with no non-trivial $\Q$-characters, it is natural to ask
\begin{equation*}
    \text{What are the possible limits of }
    (\mu_{\rmH_n} g_n) 
    \text{ under the weak-}* \text{ topology?}
\end{equation*}
where $\mu_{\rmH_n} g_n$ denotes the push-forward of $\mu_{\rmH_n}$ by $g_n$ induced from the right $\rmG$-action on $\Gamma \bs \rmG$.
This question is interesting from the point of view of counting solutions to some homogeneous Diophantine equations (see \cite{DukRudSar93,EskMcM93,EskMozSha96} etc.) and the Andre--Oort conjecture (see \cite{UllmoYafaev14,KlinglerYafaev14} etc.). 

Building on Ratner's description of unipotent-invariant ergodic measures \cite{Rat91} and the linearization technique of Dani--Margulis \cite{DanMar93}, the following two cases have received a satisfying answer via the work of Eskin, Mozes and Shah (\cite{MozShah95,EskMozSha96}):
\begin{itemize}
    \item[1.] $\bmH_n$'s are generated by $\R$-unipotents;
    \item[2.] $\bmH_n$'s are $\Gamma$-conjugate to a fixed $\bmH$.
\end{itemize}
In both cases it is proved that after passing to a subsequence,
\begin{itemize}
    \item[a.] either $\Gamma\bs\Gamma \rmH g_n$ diverges set-theoretically;
    \item[b.] or $(\mu_{\rmH_n}g_n)$ converges to a homogeneous measure.
\end{itemize}

Let us now fix a maximal compact subgroup $\rmK$ of $\rmG$. Then $\Gamma \bs \rmG /\rmK$ admits compactifications that are finer than the one-point compactification. 
Therefore by projecting this sequence of measures to $[\mu_{\rmH_n}g_n]$ supported on $\Gamma \bs \rmG /\rmK$, there is a chance of refining the case ($a$) above.
Such a result has been obtained recently by 
Daw--Gorodnik--Ullmo (see \cite{DawGoroUll18, DawGoroUllLi19}) with the help of Li. They showed 
\begin{thm}
Let $\overline{\Gamma \bs \rmG /\rmK}^{\BS}$ be the Borel--Serre compactification of $\Gamma \bs \rmG /\rmK$. Then any limit of $[\mu_{\rmH_n}g_n]$ under the weak-$*$ topology is a homogeneous measure assuming $\bmH_n$'s are generated by $\R$-unipotents.
\end{thm}

We will review the definition of the Borel--Serre compactification later in Section \ref{secPreliminary}. For now, let us just note that the boundary, $\overline{\Gamma \bs \rmG /\rmK}^{\BS} -{\Gamma \bs \rmG /\rmK}$, is a disjoint union of finitely many ``strata" of the form 
$(\rmF\cap \Gamma)  \bs \rmF / (\rmF\cap \rmK)$ for some $\Q$-subgroup $\bmF$ of $\bmG$. A homogeneous measure here refers to a measure supported on a single stratum $(\rmF\cap \Gamma) \bs \rmF / (\rmF\cap \rmK)$  and that it is the projection of some homogeneous measure from $\rmF\cap \Gamma \bs \rmF$.

In the present paper we build upon their work and show:

\begin{thm}\label{thmMain1Sec1}
Any limit of $[\mu_{\rmH_n}g_n]$ in the Borel--Serre compactification under the weak-$*$ topology is a homogeneous measure assuming $\bmH_n \equiv \bmH$ for all $n$.
\end{thm}

As a consequence (of one version, see Theorem \ref{thmMainThmSec3} of the above theorem avoiding the condition we put on $\bmH$), we obtain the following non-divergence criterion, which was one of our original motivations.
Here we only assume that $\bmH$ is a $\Q$-subgroup and fix a nonempty bounded open subset $\calO$ of $\rmH$. Then $\calO$ supports a finite measure by restricting the Haar measure (or any other measure equivalent to this one) on $\rmH$ to $\calO$.
Let $\mu_{\calO}$ be its projection to $\Gamma \bs \rmG$.

\begin{thm}\label{thmMain2Sec1}
For every $\ep>0$ and $\delta\in(0,1)$, there exists a compact set $C$ of $\Gamma \bs \rmG$ such that 
for all $g\in \rmG$ satisfying
\begin{equation}\label{EquaCriterionLinearRep}
    \norm{ \Ad(g) v_{\rmP} } \geq \ep
\end{equation}
for all $\Q$-parabolic subgroup $\bmP$ of $\bmG$ containing $\bmH$, we have
\begin{equation*}
    \mu_{\calO}(C g^{-1}) > (1-\delta) \cdot  \mu_{\calO}(\Gamma\bs\rmG).
\end{equation*}
\end{thm}

First one fix an integral structure on $\Lie(\bmG)$ and hence on $\wedge^{l}\Lie(\bmG)$ for all $l$.
Then $v_{\rmP}$ is the unique (up to sign) primitive integral vector in $\wedge^{\dim \bmP}\Lie(\bmG)$ that lifts
$[\Lie(\bmP)]$ which naturally lives in the projective space of $\wedge^{\dim \bmP}\Lie(\bmG)$.

When $\bmH$ is generated by $\R$-unipotents, this has been proved in \cite{DawGoroUll18}, extending the work of \cite{DanMar91}.
On the other hand, if Equation (\ref{EquaCriterionLinearRep}) were true for all $\Q$-parabolic subgroups of $\bmG$, not just those containing $\bmH$, then this has been established in \cite{EskMozSha97}. 

This answers a question raised in \cite[Conjecture 2.17]{Zha19} where a special case when $\bmH$ is a maximal $\Q$-split torus was treated, which was used in a counting problem studied there. We hope this general non-divergence criterion will be helpful for other counting problem as well.

\subsubsection*{Organization of the paper}

In Section \ref{secPreliminary} we fix some notations and collect some facts from reduction  theory and Borel--Serre compactifications. In Section \ref{secOutline} we state two main theorems from which Theorem \ref{thmMain1Sec1} and \ref{thmMain2Sec1} can be deduced. In Section \ref{secSL_N} we prove our theorem for $\SL_N$ using arguments from geometry of numbers. The reduction of the general case to this one is explained in Section \ref{secGeneralcase} with the help of geometric invariant theory.

\section{Preliminaries}\label{secPreliminary}

In this section we review the basics of reduction theory and Borel--Serre compactifications. Our main references are \cite{Spr98,BorSer73,BorJi06,Bor2019,DawGoroUll18}.

We start by fixing the following data:

\begin{itemize}
    \item $\bmG$ is a connected reductive group over $\Q$ with no non-trivial $\Q$-character;
    \item $\rmG$ is defined to be the connected component of $\bmG(\R)$ containing the identity and we let $\rmG_{\Q}:= \rmG\cap \bmG(\Q)$;
    \item $\rmK$ is a maximal compact subgroup of $\rmG$ and we let $\iota_{\rmK}$ be the associated Cartan involution $\iota_{\rmK}:\rmG\to\rmG$;
    \item $\rmX$ is defined as $\rmG/\rmK$;
    \item $\Gamma$ is an arithmetic lattice contained in $\rmG$.
\end{itemize}

By \cite[Proposition 1.6]{BorSer73}, we assume that $\iota_{\rmK}$ is an algebraic involution defined over $\R$. 

\subsection{Notations for linear algebraic groups}

For a linear algebraic group $\bmH$ over $\Q$ we define the following

\begin{itemize}
    \item let $\bmX^{*}(\bmH)$ (resp. $\bmX^*(\bmH_{\overline{\Q}})$) be the commutative group of characters of $\bmH$ over $\Q$ (resp. $\overline{\Q}$);
    \item let $\bmX_{*}(\bmH)$ (resp. $\bmX_{*}(\bmH_{\overline{\Q}})$) be the commutative group of co-characters of $\bmH$ over $\Q$ (resp. $\overline{\Q}$);
    \item let ${}^{\circ}\bmH$ be the subgroup of $\bmH$ defined by (see \cite[I.1.1.]{BorSer73})
         \begin{equation*}
            {}^{\circ}\bmH:= \bigcap_{\alpha\in X^*(\bmH)} \ker \alpha^2;
          \end{equation*}
    \item $\rmH$ is the identity connected component of $\bmH(\R)$ and $\rmH_{\Q}$ denotes $\rmH\cap \bmH(\Q)$;
    \item for $h\in\rmH$ and a subgroup $\rmL$, we write ${}^{h}\rmL $ for $h\rmL h^{-1}$.
\end{itemize}

\subsection{Parabolic subgroups}
For a $\Q$-parabolic subgroup $\bmP$ of $\bmG$, which is automatically connected \cite[6.4.10]{Spr98}, we define the following data:

\begin{itemize}
    \item let $\bmU_{\bmP}$ denote the unipotent radical of $\bmP$;
    \item let $\bmL_{\bmP}:= \bmP/\bmU_{\bmP}$, a connected reductive group;
    \item let $\bmS_{\bmP}$ be the maximal $\Q$-split torus contained in the center of $\bmL_{\bmP}$;
    \item let $\bmM_{\bmP}:= ({}^{\circ}\bmL_{\bmP})^{\circ}$. So $\bmL_{\bmP}= \bmM_{\bmP}\cdot \bmS_{\bmP}$ is an almost direct product.
\end{itemize}

 $\bmS_{\bmP}$ acts on the Lie algebra of $\bmP$ by choosing a lift $\widetilde{\bmS}_{\bmP}$. Different choices of lifts give actions that are conjugate to each other by $\bmU_{\bmP}$. Also define:

\begin{itemize}
     \item let $\Phi(\bmS_{\bmP},\bmP)$ be the set of non-zero characters appearing in this representation;
     \item let $\Delta(\bmS_{\bmP},\bmP)$ be the subset of  $\Phi(\bmS_{\bmP},\bmP)$ which generates it via $\Z_{\geq 0}$-linear combinations;
     \item there is a natural order-preserving bijection between
     \begin{equation*}
     \begin{aligned}
            \left\{  \text{subsets of } \Delta(\bmS_{\bmP},\bmP) \right\}
         &\longleftrightarrow
         \left\{
         \Q\text{-parabolic subgroups of }\bmG \text{ containing } \bmP \right\}
         \\
         I &\longleftrightarrow \bmP_I
     \end{aligned}
     \end{equation*}
     where the empty set corresponds to $\bmP$ and the full $\Delta$ corresponds to $\bmG$;
     \item let $\bmS_{\bmP,I}:=\left( \cap_{\alpha\in I} \ker \alpha
     \right)^{\circ}$, then $\bmP_I$ is generated by $\bmZ_{\bmG}(\widetilde{\bmS}_{\bmP,I})$ and $\bmU_{\bmP}$ for any lift $\widetilde{\bmS}_{\bmP,I}$.
\end{itemize}

\subsection{Horospherical coordinates}

Given a parabolic subgroup $\bmP$ of $\bmG$ as in the last section, sections of $\pi_{\bmP}: \bmP \to \bmL_{\bmP}$ are not unique but there exists a unique (see \cite[Corollary 1.9]{BorSer73}) $i_{\rmK}:\bmL_{\bmP} \to \bmP$ defined over $\R$ such that $ \pi_{\bmP} \circ i_{\rmK} =id$ and its image $\bmL_{\bmP,\rmK}$ is stable under $\iota_{\rmK}$.

\begin{equation*}
    \begin{tikzcd}
      \bmP \arrow[d, "\pi_{\bmP}"'] \\
      \bmL_{\bmP} \arrow[u, bend right=70, "i_{\rmK}"']
\end{tikzcd}
\end{equation*}

Now we define
\begin{itemize}
    \item $\bmA_{\bmP,\rmK}:= i_{\rmK}(\bmS_{\bmP})$, $\bmM_{\bmP,{\rmK}}:= i_{\rmK}(\bmM_{\bmP})$  and $\bmL_{\bmP,\rmK}:=i_{\rmK}(\bmL_{\bmP})$.
\end{itemize}

As $\rmK$ will be fixed, we shall drop the dependence on $\rmK$ if no confusion might arise. 

We have a natural map by taking multiplications:
\begin{equation*}
    \begin{aligned}
           \nu:  \rmU_{\bmP}\times \rmA_{\bmP} \times  \rmM_{\bmP}\rmK  \to \rmG 
    \end{aligned}
\end{equation*}

\begin{lem} \cite[Proposition 1.5]{BorSer73}
$\nu$ is a homeomorphism.
\end{lem}

In light of this, for $g\in \rmG$ we write 
\begin{equation*}
    g= u_{\rmP}(g) \cdot a_{\rmP}(g) \cdot m_{\rmP}(g) \cdot k_{\rmP}(g)
\end{equation*}
with the understanding that $ m_{\rmP}(g) \cdot k_{\rmP}(g)$ is uniquely determined. We shall refer to this as the \textbf{horospherical coordinate} of $g$ associated with $\bmP$ and $\rmK$.

\subsection{Relation between horospherical coordinates}\label{secRelationHoro1}

Given two $\Q$-parabolic subgroups $\bmP\subset\bmP'$ in $\bmG$, 
we can write $\bmP'= \bmP_I$ for some unique $I\subset \Delta(\bmS_{\bmP},\bmP)$.
Then $\bmA_{\bmP,I} = \bmA_{\bmP_I}$ 
(see \cite[III.1.16]{BorJi06}) and 
\begin{itemize}
    \item let $\bmA^I_{\bmP}$ be the $\R$-subtorus of $\bmA_{\bmP}$ defined by taking orthogonal complement of $\bmA_{\bmP,I}$ with respect to the killing form.
\end{itemize}   
Note that $\bmA^I_{\bmP}$ is contained in $\bmM_{\bmP_I,\rmK}$.
\begin{itemize}
    \item Let $\bmU_{\bmP}^I$ be the unipotent $\R$-subgroup defined by $\bmU_{\bmP}\cap \bmM_{\bmP_I,\rmK}$.
\end{itemize}
Then one has $\bmU_{\bmP}=\bmU_{\bmP_I} \rtimes \bmU_{\bmP}^I$ and consequently we have the natural homeomorphism
\begin{equation}\label{equaU_PrelatetoU_P'}
    \rmU_{\bmP}= \rmU_{\bmP_I}\times \rmU_{\bmP}^I .
\end{equation}
Additionally we have the following 
\begin{equation}\label{equaA_PrelatetoA_P'}
    \rmA_{\bmP} = \rmA_{\bmP_I} \times \rmA_{\bmP}^I 
\end{equation}
and 
\begin{equation*}
    \rmM_{\bmP,\rmK} \subset \rmM_{\bmP_I,\rmK}.
\end{equation*}

Consequently, starting from 
\begin{equation*}
    g= u_{\bmP}(g) a_{\bmP}(g) m_{\bmP}(g) k_{\bmP}(g),
\end{equation*}
we write $u_{\bmP}(g) = u_{\bmP,I}(g)u^I_{\bmP}(g)$ according to Equation (\ref{equaU_PrelatetoU_P'}) and $a_{\bmP}(g)=a_{\bmP,I}(g) a^I_{\bmP}(g)$ according to Equation (\ref{equaA_PrelatetoA_P'}). Then
\begin{equation*}
\begin{aligned}
         g =&  (u_{\bmP,I}(g)u^I_{\bmP}(g))\cdot (a_{\bmP,I}(g) a^I_{\bmP}(g))\cdot (m_{\bmP}(g) k_{\bmP}(g))\\
        =& u_{\bmP,I}(g) \cdot a_{\bmP,I}(g)\cdot 
         (
         u^I_{\bmP}(g)a^I_{\bmP}(g)m_{\bmP}(g) k_{\bmP}(g)
         ).
\end{aligned}
\end{equation*}
And therefore
\begin{itemize}
    \item $u_{\bmP_I}(g)=u_{\bmP,I}(g)$;
    \item $a_{\bmP_I}(g)=a_{\bmP,I}(g)$;
    \item $m_{\bmP_I}(g)k_{\bmP_I}(g)=u^I_{\bmP}(g)a^I_{\bmP}(g)m_{\bmP}(g) k_{\bmP}(g)$.
\end{itemize}

\subsection{Relative horospherical coordinates}

Let $\bmP \subset \bmP'=\bmP_I$ be as in the last section. Then $\overline{\bmP}:=\pi_{\bmP_I}(\bmP)\cap \bmM_{\bmP_I}$ is a $\Q$-parabolic subgroup of $\bmM_{\bmP_I}$ and $\overline{\rmK}:=\pi_{\bmP_I}{(\rmK\cap \rmP_I)}$ is a maximal compact subgroup of $\rmM_{\bmP_I}$. Hence one may talk about the horospherical coordinate of $\rmM_{\bmP_I}$ associated with  $\overline{\bmP}$ and $\overline{\rmK}$. 

Observe the following:
\begin{itemize}
    \item $\bmA_{\overline{\bmP},\overline{\rmK}} = \pi_{\bmP_I}(\bmA^I_{\bmP})$ and $i_{\rmK}( \pi_{\bmP_I}(\bmA^I_{\bmP})) = \bmA_{\bmP}^I$;
     \item $\bmU_{\overline{\bmP}}  = \pi_{\bmP}(\bmU^I_{\bmP})$ and $i_{\rmK}( \pi_{\bmP_I}(\bmU^I_{\bmP})) = \bmU_{\bmP}^I$;
     \item  $\bmM_{\overline{\bmP},\overline{\rmK}}= \pi_{\bmP_I}(\bmM_{\bmP,\rmK})$ and $i_{\rmK}( \pi_{\bmP_I}(\bmM_{\bmP,\rmK})) = \bmM_{\bmP,\rmK}$;
     \item $\rmK \cap \rmP_I$ is contained in $\rmM_{\bmP_I,\rmK}$, so $i_{\rmK}(\overline{\rmK})=\rmK \cap \rmM_{\bmP_I,\rmK}$.
\end{itemize}

Hence the horospherical coordinate is
\begin{equation*}
    \rmM_{\bmP_I} = \pi_{\bmP_I}(\rmU^I_{\bmP}) \times 
    \pi_{\bmP_I} (\rmA^I_{\bmP})
    \times \pi_{\bmP_I}(\rmM_{\bmP,\rmK}) \overline{\rmK}
\end{equation*}
and by applying $i_{\rmK}$ we get
\begin{equation*}
    \rmM_{\bmP_I,\rmK} =
    \rmU^I_{\bmP}\times 
    \rmA^I_{\bmP}
    \times 
    \rmM_{\bmP,\rmK} (\rmK \cap \rmM_{\bmP_I,\rmK}).
\end{equation*}
Such isomorphisms shall be referred to as \textbf{relative horospherical coordinates}.

Define 
\begin{itemize}
    \item $\rmX_{\bmP}:=\rmM_{\bmP,\rmK}\rmK/\rmK$, $\rmX:= \rmG/\rmK$.
\end{itemize}

One is referred to the following commutative diagram for a comparison with the the last section.

\begin{equation*}
    \begin{tikzcd}
          \rmU_{\bmP_I}\times \rmA_{\bmP_I} \times \rmX_{\bmP_I}
          \arrow[r,leftrightarrow, "\text{horospherical}"]\arrow[d,leftrightarrow,"\text{relative}"',"\text{horospherical}"]
          & \rmG/\rmK\arrow[r,leftrightarrow,"\text{horospherical}"] 
          & \rmU_{\bmP}\times \rmA_{\bmP} \times \rmX_{\bmP}
          \arrow[d,leftrightarrow,"\text{last}"',"\text{section}"]
          \\
          \rmU_{\bmP_I}\times \rmA_{\bmP_I} \times \rmU^I_{\bmP}\times
          \rmA^I_{\bmP}
          \times \rmX_{\bmP}
          \arrow[rr, "\text{permutate}", "\text{coordinates}"',leftrightarrow] 
          &&\rmU_{\bmP_I}\times  \rmU^I_{\bmP} \times \rmA_{\bmP_I} \times \rmA^I_{\bmP} \times \rmX_{\bmP_I}
    \end{tikzcd}
\end{equation*}

\subsection{Partial Borel--Serre compactification}

For each $\Q$-parabolic subgroup $\bmP$ of $\bmG$, we define
\begin{equation*}
    e(\bmP):= \rmU_{\bmP} \times \rmX_{\bmP}.
\end{equation*}
in terms of the horospherical coordinates by erasing the $\rmA_{\bmP}$-terms.
For instance $e(\bmG)=\rmX$.

The \textbf{partial Borel--Serre compactification} of $\rmX$, as a set, is defined to be
\begin{equation*}
    {}_{\Q}\overline{\rmX}^{\BS}= \rmX \sqcup \bigsqcup e(\bmP)
\end{equation*}
as $\bmP$ ranges over all proper $\Q$-parabolic subgroups of $\bmG$.

For a pair $\bmP\subset \bmP_I$ of $\Q$-parabolic subgroups of $\bmG$, by the relative horospherical coordinates from the last section we have
\begin{equation*}
    e(\bmP_I)= \rmU_{\bmP} \times \rmA_{\bmP}^I \times \rmX_{\bmP}.
\end{equation*}

We equip ${}_{\Q}\overline{\rmX}^{\BS}$ with the unique topology such that the following is true 
(see \cite[III.9.2]{BorJi06}).

Under the horospherical coordinates associated with $(\bmP,\rmK)$, a sequence  $((n_i,a_i,x_i))$ of $\rmX$
converges to $(n_{\infty},x_{\infty})\in e(\bmP)$ if and only if
\begin{itemize}
    \item $\alpha(a_i)\to \infty$ for all $\alpha\in\Delta(\bmA_{\bmP,\rmK},\bmP)$ (equiv. for all $\alpha\in\Phi(\bmA_{\bmP,\rmK},\bmP)$);
    \item $(n_i)$ converges to $n_{\infty}$;
    \item $(x_i)$ converges to $x_{\infty}$.
\end{itemize}

Under the relative horospherical coordinates above, a sequence $((n_i,a_i,x_i))$ from $e(\bmP_I)$ converges to $(n_{\infty},x_{\infty})$ in $e(\bmP)$ if and only if
\begin{itemize}
     \item $\alpha(a_i)\to \infty$ for all
     $\alpha\in I$ (equiv. for all $\alpha \in \Phi(\bmA_{\overline{\bmP},\overline{\rmK}}, \overline{\bmP})$);
    \item $(n_i)$ converges to $n_{\infty}$;
    \item $(x_i)$ converges to $x_{\infty}$
\end{itemize}
where the notation $\overline{\bmP}$ is taken from the last section.

\subsection{Borel--Serre compactifications}
In this subsection we describe how $\rmG_{\Q}$-action naturally intertwines different horospherical coordinates.
By forgetting the $\rmA_{\bmP,\rmK}$-components, this defines an action of $\rmG_{\Q}$ on ${}_{\Q}\overline{\rmX}^{\BS}$.

Fix a parabolic subgroup $\bmP$ over $\Q$. 
For $g\in \bmG(\R)$ we let $c_g: \rmP \to {}^{g}\rmP$ be defined by $c_g(p):=gpg^{-1}$ and it descends to a morphism between $\rmL_{\bmP}\to \rmL_{{}^{g}\bmP}$, which we denote by $\overline{c}_g$. So we have the following commutative diagram
\begin{equation*}
    \begin{tikzcd}
          \rmP \arrow[r, "c_g"]\arrow[d, "\pi_{\bmP}"'] & 
          {}^{g}\rmP \arrow[d, "\pi_{{}^{g}\bmP}"]\\
          \rmL_{\bmP}\arrow[r, "\overline{c}_g"]  & \rmL_{{}^{g}\bmP}.
    \end{tikzcd}
\end{equation*}
If $g$ is contained in $\bmG(\Q)$ then $\overline{c}_g$ carries $\rmS_{\bmP}$ over to $\rmS_{{}^{g}\bmP}$ by definition. Now we observe that this holds more generally for $g \in \bmG(\R)$ as long as ${}^{g}\bmP$ is defined over $\Q$ (otherwise, $\rmS_{{}^{g}\bmP}$ would not be defined anyway).

To see this recall that $\bmP$ and ${}^{g}\bmP$ are actually conjugate over $\bmG(\Q)$ (actually also true over $\rmG_{\Q}$) and the normalizer of a parabolic subgroup is equal to itself, so we can find $p\in {}^{g}\bmP(\R)$ and $q \in \bmG{(\Q)}$ such that $g=pq$. As $\bmS_{{}^{g}\bmP}$ is contained in the center of $\bmL_{{}^g\bmP}$, we have $\overline{c}_{p}(\rmS_{{}^{g}\bmP}) = \rmS_{{}^{g}\bmP}$. 
It follows that 
\begin{equation*}
    \overline{c}_g(\rmS_{\bmP})
    =\overline{c}_{p}\circ \overline{c}_{q} (\rmS_{\bmP}) = \rmS_{{}^{g}\bmP}.
\end{equation*}

In particular this holds for $g=k\in \rmK$. 
Thus $k\rmA_{\bmP,\rmK}k^{-1}= \rmA_{{}^{k}\bmP, \rmK}$ and consequently 
$k\rmM_{\bmP,\rmK}k^{-1}= \rmM_{{}^{k}\bmP, \rmK}$ also holds.

Now we are ready to describe the $\rmG_{\Q}$-action. So let us fix $q\in \rmG_{\Q}$ and write, for simplicity, the horospherical coordinate of $q$ with respect to (${}^{q}\bmP$, $\rmK$) as $q=(u_q,a_q,m_qk_q)$.

Then the following diagram commutes:
\begin{equation*}
    \begin{tikzcd}
          \rmG \arrow[r,"\text{left multiplication by }q"]
          \arrow[d]
          & \rmG \arrow[d]\\
          \rmU_{\bmP} \times \rmA_{\bmP} \times \rmM_{\bmP}\rmK
          \arrow[r,"q \cdot "]
          &  \rmU_{{}^{q}\bmP} \times \rmA_{{}^{q}\bmP} \times \rmM_{{}^{q}\bmP}\rmK
    \end{tikzcd}
\end{equation*}
where the bottom arrow is defined by $(u,a,mk)\mapsto q\cdot(u,a,mk)$ with
\begin{equation*}
    \begin{aligned}
         q\cdot(u,a,mk)
         = \left(
         u_q {}^{a_qm_qk_q}u,a_q {}^{k_q}a, m_q {}^{k_q}m k_q k
         \right).
    \end{aligned}
\end{equation*}

This action is a proper and continuous action (see \cite[III.9.15,17]{BorJi06}).
The quotient
$\Gamma \bs{}_{\Q}\overline{\rmX}^{\BS}$
is compact and is called the \textbf{Borel--Serre compactification} of $\Gamma\bs\rmX$.

Set theoretically 
\begin{equation*}
    \Gamma \bs{}_{\Q}\overline{\rmX}^{\BS} 
    = \bigsqcup_{[\bmP]} \Gamma \cap \rmP \bs e(\bmP)
\end{equation*}
where the disjoint union runs over all $\Gamma$-conjugacy classes of $\Q$-parabolic subgroups of $\bmG$.

\subsection{Siegel sets}

Start with the following data:
\begin{itemize}
    \item let $\bmP_0$ be a $\Q$-parabolic subgroup of $\bmG$;
    \item let $\omega_{\rmM}$ be a compact subset of $\rmM_{0}:=\rmM_{\bmP_0,\rmK}$;
    \item  let $\omega_{\rmU}$ be a compact subset of $\rmU_{0}:=\rmU_{\bmP_0}$;
    \item let $t>0$ be a positive number and let $\rmA_0:=\rmA_{\bmP_0,\rmK}$,
\end{itemize}
we define the (normal) \textbf{Siegel set} $\frakS_{\omega,t}$ associated with $(\bmP_0,\rmK)$ in terms of the horospherical coordinates associated to $(\bmP_0,\rmK)$ by
\begin{equation}
    \frakS_{\omega,t} := \omega_{\rmU} \times \rmA_t \times \omega_{\rmM}\rmK
\end{equation}
where 
\begin{equation*}
    \rmA_t := \left\{
    a\in \rmA_0 \;\mid\;
    \alpha(a)> t, \; \forall\alpha\in \Delta(\rmA_0,\bmP_0)
    \right\}.
\end{equation*}

When we say that $\frakS$ is \textit{\textbf{a}} Siegel set associated with $(\bmP_0,\rmK)$, it is understood that an implicit choice of $\omega_{\rmM}, \omega_{U}$ and $t>0$ is made and $\frakS=\frakS_{\omega,t}$.

The main theorem of reduction theory (see \cite[Theorem 13.1, 15.5]{Bor2019}) is 
\begin{thm}\label{thmReducThy}
Let $\bmP_0$ be a minimal $\Q$-parabolic subgroup of $\bmG$.
\begin{enumerate}
    \item There exists a finite set $F\subset \rmG_{\Q}$ and a Siegel set $\frakS_{\omega, t}$ associated to $(\bmP_0,\rmK)$ such that 
    \begin{equation*}
        \rmG= \Gamma \cdot F \cdot \frakS_{\omega,t};
    \end{equation*}
    \item For any Siegel set $\frakS_{\omega, t}$ associated to $(\bmP_0,\rmK)$ and any $q_1, q_2 \in \rmG_{\Q}$ we have
    \begin{equation*}
        \#
        \left\{
        \gamma \in \Gamma \,\mid\,
        \gamma q_1 \frakS_{\omega,t} \cap q_2 \frakS_{\omega,t} \neq \emptyset
        \right\} <\infty.
    \end{equation*}
\end{enumerate}
\end{thm}

In \cite{Bor2019}, it was assumed that $\bmA_{\bmP_0,\rmK}$ is a $\Q$-split torus. But the above theorem can be extended lifting this assumption with the help of  \cite[Proposition III.1.11]{BorJi06}. Indeed for a fixed $\Q$-parabolic subgroup $\bmP$, the Siegel sets associated with $(\bmP,\rmK')$ for a different maximal compact subgroup $\rmK'$ of $\rmG$ only differ from those associated with $(\bmP,\rmK)$ by multiplying an element of $\rmU_{\bmP}$ from the \textit{right}.

\section{Outline of the proof}\label{secOutline}

Let us fix the following data in this section:
\begin{itemize}
    \item $\bmG$ is a connected semisimple algebraic group over $\Q$;
    \item $\rmK$ is a maximal compact subgroup of $\rmG$ and $\rmX:=\rmG/\rmK$ is the associated symmetric space;
    \item $\Gamma\subset \rmG_{\Q}$ is an arithmetic lattice;
    \item $\bmH$ is a connected $\Q$-subgroup of $\rmG$;
    \item $\calO$ is a nonempty bounded open subset of $\rmH$;
    \item $\wtmu_H$ is a smooth measure with full support on $\rmH$ and $\wtmu_{\calO}$ denotes its restriction to $\calO$;
    \item we denote by $\mu_{\calO}$ the push-forward measure of $\wtmu_{\calO}$ on $\Gamma\bs\rmG$ and $[\mu_{\calO}]$ its further push-forward on $\Gamma\bs\rmX$;
    \item $(g_n)$ is a sequence in $\rmG$.
\end{itemize}

By the result of \cite{EskMozSha96} and the argument presented in \cite{DawGoroUll18}, Theorem \ref{thmMain1Sec1} is reduced to the following (see \cite[Theorem 5.1]{DawGoroUll18}, \cite[Theorem 4.2]{DawGoroUllLi19} and the paragraph right above \cite[Proposition 5.3]{DawGoroUllLi19})

\begin{thm}\label{thmMainThmSec3}
After passing to a subsequence, there exists 
\begin{itemize}
    \item a $\Q$-parabolic subgroup $\bmP$ of $\bmG$;
    \item a sequence $(\lambda_n)$ in $\Gamma$,
\end{itemize}
such that 
\begin{itemize}
    \item $\bmP$ contains $\lambda_n^{-1}\bmH\lambda_n$ for all $n$
\end{itemize}
and if we write 
\begin{equation*}
    \lambda_n^{-1} g_n =  u_n a_n m_nk_n,
\end{equation*}
the horospherical coordinates with respect to $(\bmP,\rmK)$, then
\begin{itemize}
    \item $\alpha (a_n)\to \infty$ for all $\alpha\in\Delta(\bmA_{\bmP,\rmK},\bmP)$;
    \item $(\lambda_n^{-1}\mu_{\calO}\lambda_n u_n m_n)$ is non-divergent in 
    $\Gamma\cap \rmP \bs\rmP$.
\end{itemize}
\end{thm}

This is proved in \cite{DawGoroUllLi19} when $\bmH$ is generated by $\R$-unipotents where they also allow $\bmH$ to vary. Actually their proof also works in our case provided a non-divergence criterion is proven. Let us now turn to this. 

First we give some definitions:
\begin{itemize}
    \item let $\bmF$ be a connected reductive linear algebraic group over $\Q$ and $\bmD$
    be a connected linear algebraic $\Q$-subgroup;
    \item $\rmK_{\rmF}$ is a maximal compact subgroup of $\rmF$ and $\Gamma_{\rmF} \subset \rmF_{\Q}$ is an arithmetic subgroup;
    \item fix an $\Ad(\rmK_{\rmF})$-invariant Euclidean norm on $\Lie(\bmF)_{\R}$,
    which induces a $\Ad(\rmK_{\rmF})$-invariant Euclidean norm on $\wedge^l \Lie(\bmF)_{\R}$ for every $l$;
    \item fix an $\Gamma_{\rmF}$-stable integral structure  on $\Lie(\bmF)$, which induces an $\Gamma_{\rmF}$-stable integral structure on $\wedge^l \Lie(\bmF)$ for every $l$;
    \item for a $\Q$-subgroup $\bmP$, let $v_{\rmP}$ be the unique (up to $\pm1$) integral vector in $\wedge^{\dim \bmP} \Lie(\bmF)$ that represents $\bmP$. 
    That is to say, in addition to being integral,
    $v_{\rmP}$ is also required to be proportional to $v_1\wedge...\wedge v_{\dim \bmP}$ for a basis $\{v_1,...,v_{\dim \bmP}\}$ of $\Lie(\bmP)$;
    \item for $g\in \rmF$ and a $\Q$-subgroup $\bmP$ of $\bmF$, define 
    \begin{equation*}
        l_{\bmP,\rmK_{\rmF}}(g):= \norm{g\cdot v_{P}}
    \end{equation*}
    \item for $g\in \rmF$ define 
    \[
    \delta_{\bmF,\rmK_{\rmF},\bmD}(g):= 
    \inf \left\{
     l_{\bmP,\rmK_{\rmF}}(g^{-1}) 
       \:\middle\lvert\:
    \parbox{60mm}
    {
    \raggedright 
    $\bmP $ is a maximal proper $\Q$-parabolic subgroup of  $\bmF$ containing $\bmD$
    }
    \right\}.
    \]
\end{itemize}

Our $l_{\bmP,\rmK_{\rmF}}$ function is very different from  the $d_{\bmP,\rmK_{\rmF}}$ function as presented in \cite[Section 3.3]{DawGoroUllLi19}. 
However, there are only finitely many $\Q$-parabolic subgroups up to $\Gamma_{\rmF}$-conjugacy.
Therefore our 
$\delta_{\bmF,\rmK_{\rmF},\bmD}$ only differs from that of \cite[Section 3.4]{DawGoroUllLi19} by a bounded error. 
Also note that for $\gamma \in \Gamma_{\rmF}$, $\gamma \cdot v_{\rmP} = \pm1 v_{\gamma \bmP \gamma^{-1}}$ for a $\Q$-subgroup $\bmP$. 
So this function descends to $\Gamma_{\rmF}\bs \rmF$.

Now we go back to the set-up of Theorem \ref{thmMainThmSec3}. For a parabolic $\Q$-subgroup $\bmP$ (including the possibility of $\bmP=\bmG$), recall that $\pi_{\bmP}$ is the projection $\bmP \to \bmP/\bmU_{\bmP}$. 

\begin{thm}\label{thmNondivCriterionSec3}
We let $\overline{\Gamma}:= \pi_{\bmP}(\Gamma\cap \rmP)$ and $\overline{\rmK}:=\pi_{\bmP}(\rmK\cap \rmP)$.
For every sequence $(g_n)$ in $\rmU_{\bmP}\cdot \rmM_{\bmP,\rmK}$ and $(\lambda_n)$ of $\Gamma$ such that $\lambda_n^{-1} \bmH \lambda_n$ is contained in $\bmP$ and 
\begin{equation*}
    \liminf_{n\to\infty}
    \delta_{\bmL,\overline{\rmK},\overline{\bmH}_n}(\pi_{\bmP}(g_n)) >0,
\end{equation*}
where we write $\overline{\bmH}_n:= \pi_{\bmP}(\lambda_n^{-1} \bmH \lambda_n)$, 
we have $([\lambda_n^{-1} \mu_{\calO}\lambda_n g_n])$ is sequentially compact on $\Gamma_{\rmP}\bs \rmP$.
\end{thm}

Theorem \ref{thmMain2Sec1} is the case when $\bmG=\bmP$.

The proof presented in \cite[Section 5.3]{DawGoroUllLi19} actually shows that if Theorem \ref{thmNondivCriterionSec3} is true for all $\Q$-parabolic subgroup $\bmP$ of $\bmG$, then Theorem \ref{thmMainThmSec3} is true. 
Moreover Theorem \ref{thmNondivCriterionSec3} has been verified in \cite[Theorem 4.6]{DawGoroUll18} when $\bmH$ is generated by $\R$-unipotents, extending previous results of \cite{DanMar91}.

However, we will only be able to prove Theorem \ref{thmNondivCriterionSec3} directly when $\bmG=\SL_N$. By the last paragraph this implies Theorem \ref{thmMainThmSec3} for $\SL_N$. For general $\bmG$, we choose an embedding into some $\SL_N$. What we show is that we can deduce Theorem \ref{thmMainThmSec3} for $\bmG$ from the statement for $\SL_N$. Note that Borel--Serre compactifications do not have ``functorialities'', so this is a non-trivial task. As Theorem \ref{thmMainThmSec3} formally implies Theorem \ref{thmNondivCriterionSec3} we finally get it for all $\bmG$ and $\bmP$.

\section{Case of SL$_N$}\label{secSL_N}

\subsection{Non-divergence criterion in general}

In this section we recall some non-divergence criterion presented in \cite{DanMar91,KleMar98,EskMozSha97}.

Let $\bmH$ be a $\Q$-subgroup of $\SL_N$ and $\calO\subset \rmH$ be a nonempty open bounded subset.
Without loss of generality we assume that $\calO= (h_0 \exp{(O)})^{-1}$ for some $h_0\in \rmH$ and $O \subset \Lie(\rmH)$.
We equip $\calO$ with a smooth measure and push forward this to a measure $[\mu_{\calO}]$ supported on $\rmY:=\SL_N(\Z)\bs\SL_N(\R)/\SO_N(\R)$.

For two positive number $C,\alpha$, we say that a function $f: O \to \R$ is $(C,\alpha)$-\textbf{good} (see \cite[Section 3]{KleMar98}) iff for all open balls $B\subset O$ and $\ep>0$, one has
\begin{equation*}
    \frac{1}{\Leb (B)} \Leb
    \left\{ 
     x\in B \,\vert\,
     |f(x)|\leq \ep
    \right\}
    \leq C \cdot \left( 
    \frac{\ep}{\sup_{x\in B} |f(x)|}
    \right)^{\alpha}.
\end{equation*}

Let us fix a sup-norm $\norm{\cdot} _{\sup}$ on $\R^N$ and all of its exterior powers. For each $l$-dimensional $\Q$-linear subspace $W$ of $\Q^N$, let $w_1,...,w_l$ be a basis of $W\cap \Z^N$ and we define $\Lambda_W:=w_1\wedge ...\wedge w_l$. This is independent of the choice of the basis up to $\pm1$.
Then it is known (see \cite[Section 3]{EskMozSha97}, \cite[Section 3]{KleMar98}) that there exists $C,\alpha>0$ such that for all $\Q$-subspaces $W$, for all $g,\gamma \in\rmG$, the map
\begin{equation}\label{equaDefiC,AlphaGood}
    x \mapsto \norm{g^{-1}\gamma h_0 \exp(x) \gamma^{-1}\Lambda_W}_{\sup}
\end{equation}
is a $(C,\alpha)$-good function on $O$.

A theorem of Kleinbock--Margulis \cite[Theorem 5.2]{KleMar98} (compare \cite[Theorem 3.5]{EskMozSha97}) implies that 
\begin{thm}\label{thmKleinbockMargulis}
For any $\delta,\ep>0$, there exists a compact set $C$ of $\rmY$  such that for every $g,\gamma \in \rmG$ satisfying
\begin{equation*}
    \sup_{x\in O} \norm{g^{-1}\gamma h_0 \exp(x) \gamma^{-1}\Lambda_W} \geq \ep
\end{equation*}
for all rational subspaces $W$ of $\Q^N$, we have
\begin{equation*}
    [\mu_{\calO}]
    \left\{
    x\in \calO \,\vert\,
    [\gamma o \gamma^{-1} g] \notin C
    \right\}
    \leq \delta [\mu_{\calO}](\rmY).
\end{equation*}
\end{thm}

\subsection{Refined non-divergence criterion for SL$_N$}

In this section we prove Theorem \ref{thmNondivCriterionSec3} for $\bmG=\SL_N$. Before the proof let us note two things.

First, there is a bijection between 
\begin{equation*}
\begin{aligned}
     &\{\text{maximal proper }\Q\text{-parabolic subgroups of } \SL_N  \}\\
    &\longleftrightarrow
    \{
    \text{non-trivial }\Q\text{-subspaces of }\Q^N
    \}
\end{aligned}
\end{equation*}
via taking the stabilizer of a $\Q$-subspace.
Moreover, if a maximal $\Q$-parabolic subgroup $\bmP$ corresponds to a rational subspace $W$, then up to a bounded error,
\begin{equation*}
    l_{\bmP,\SO_N(\R)}(g) =
    \norm{g\cdot v_{\rmP}}
    \approx \norm{g \cdot \Lambda_W}.
\end{equation*}

Second, in the set-up of Theorem \ref{thmNondivCriterionSec3} when $\bmG=\SL_N$, all possible $\bmL_{\bmP}$'s are products of some $\SL_{n_i}$'s and $\bmG_m$'s. 
So up to a finite cover,
\begin{equation*}
    \Gamma \cap  \rmP\bs \rmP \approx \prod \SL_{n_i}(\Z)\bs \SL_{n_i}(\R) \times \prod \R_{>0}.
\end{equation*}
Let $\pi_{\Split}$ be the projection from $\bmP$ to $\prod \bmG_m$, then every $g_n$ as in the theorem vanishes and 
\begin{equation*}
    h\mapsto \pi_{\Split}(\lambda_n^{-1} h\lambda_n)
\end{equation*}
 descends to morphisms between tori. Therefore by rigidity of diagonalizable groups \cite[Proposition 3.2.8]{Spr98}, after passing to a subsequence, these morphisms are independent of $n$. 
Hence $(\lambda_n^{-1} \mu_{\calO}\lambda_n g_n)$ is non-divergent in $\Gamma \cap \rmP \bs \rmP$ iff 
$(\overline{\pi}_i(\lambda_n^{-1} \mu_{\calO}\lambda_n g_n))$ is so for every projection 
$\overline{\pi}_i : \Gamma \cap \rmP \bs \rmP \to \SL_{n_i}(\Z)\bs \SL_{n_i}(\R)$. Write $\pi_i: \bmP \to \SL_{n_i}$ for the corresponding projection for groups.

Let $\bmX(\bmH,\bmP):=\{g\in\bmG,\, g\bmH g^{-1}\subset \bmP\}$ and for every $i$, $\gamma \in \bmX(\bmH,\bmP)\cap \Gamma$, define $p^i_{\gamma}: \bmH \to \SL_{n_i}$ by
\begin{equation*}
   p^i_{\gamma}:\, h \mapsto \pi_i(\gamma h \gamma^{-1}).
\end{equation*}
Then up to a bounded error, for any $\gamma \in \bmX(\bmH,\bmP) \cap \Gamma $ and $g \in \rmU_{\bmP}\cdot \rmM_{\bmP}$,
\begin{equation*}
    \delta_{\bmL,\overline{\rmK},\overline{\gamma^{-1}\bmH\gamma}}(g)
    \approx 
    \min_{i} \inf
    \left\{
     \norm{g^{-1}\Lambda_W} \:\middle\vert\:
     W\leq_{\Q} \Q^{n_i}, \,p^i_{\gamma}(\rmH)\text{-stable subspaces}
    \right\}.
\end{equation*}

Therefore to prove Theorem \ref{thmNondivCriterionSec3} for $\SL_N$, in light of Theorem \ref{thmKleinbockMargulis}, it suffices to prove the following (see \cite[Proposition 2.4]{Zha19} for a special case).  From now on we fix some $i=i_0$ and write $p_{\gamma}=p^{i_0}_{\gamma}$, $n=n_{i_0}$ for simplicity.

\begin{prop}\label{propNonDivCrit}
For every $\eta>0$ there exists $\ep>0$ such that for every $g\in \SL_n(\R)$, $\gamma \in \rmX(\bmH,\bmP)\cap \Gamma$ satisfying
\begin{equation*}
    \norm{g\Lambda_W}\geq \eta
\end{equation*}
for every non-trivial $p_{\gamma}(\bmH)$-stable $\Q$-subspace $W$, 
then we have 
\begin{equation}\label{equaProp4.2}
    \sup_{o\in \calO^{-1}}\norm{gp_{\gamma}(o)\Lambda_W}\geq \ep
\end{equation}
for every non-trivial rational subspace $W$.
\end{prop}

Before the actual proof, let us make some preparations. Recall that $\calO = (h_0\exp(O))^{-1}$ for some nonempty open bounded set $O$ in $\Lie(\rmH)$ and some $h_0\in\rmH$. Without loss of generality, assume that $O$ contains $0$. Let $O_{\Q}:=\{o\in O\,\vert\, \exp(o)\in \bmH(\Q)\}$.

First, choose $O'$ small such that $(\exp{ O'})^n$ is contained in $\exp(O)$. 
As the measure $\mu_{\calO}$ is assumed to be smooth, when identifying it with a measure on $O$, we may and do assume that this measure is the standard Lebesgue measure.

Second, recall that there exist two constants $C,\alpha>0$ such that for every $g\in\SL_n(\R)$, $\gamma \in \rmX(\bmH,\bmP)\cap \Gamma$, $h_0 \in\rmH$ and every $\Q$-linear subspace $W$ of $\Q^n$,
\begin{equation*}
    o\mapsto \norm{gp_{\gamma}(h_0\exp(o))\Lambda_W}_{\sup}
\end{equation*}
defines a $(C,\alpha)$-good function on $O'$.

In the argument below, we will need to switch between the sup-norm and the Euclidean norm. Each time a bounded error will be created. However this will only be done in finitely many times (depending on $n$), so we will be ambiguous about the choice of the norm in the proof.

Third, let $V$ be the direct products of all exterior powers of $\Q^n$. Consider the vector space 
\begin{equation*}
\begin{aligned}
     \scrF:=& \text{ linear span of }\\
     &\left\{
     f : O' \to \C \,\middle\vert\,
     f(o)= \la 
     p_{\gamma}(h_0 \exp(o)) v,l
     \ra,\,
     \exists\, v\in V, \, l\in V^*,\,\gamma\in \bmX(\bmH,\bmP)\cap \Gamma
    \right\}.
\end{aligned}
\end{equation*}
$\scrF$ is finite-dimensional. And the restriction $\scrF \to \scrF\vert_{O'_{\Q}}$ is an isomorphism. As this is finite-dimensional, there exists (and we fix such) a finite set $\Xi \subset O'_{\Q}$ such that $\scrF \to \scrF\vert_{\Xi}$ is an isomorphism. In particular, this means that for any $\gamma \in \bmX(\bmH,\bmP)\cap \Gamma$, 
if a subspace of $\Q^n$ is stable under $p_{\gamma}(h_0\exp(\xi))$ for all $\xi \in \Xi$, then it is stable under $p_{\gamma}(\bmH)$.
Depending on $\Xi$, there exists a positive integer $Z_0$ such that for all $\gamma \in \rmX(\bmH,\bmP)$, 
\begin{equation*}
    p_{\gamma} ( \xi)\cdot  V_{\Z} \subset \frac{1}{Z_0} \cdot V_{\Z}, \quad \forall \xi \in \Xi.
\end{equation*}

Lastly, as we shall use some argument (see Lemma \ref{lemGeometryOfNumber} below) from geometry of numbers, by abuse of notation, let us write $\Lambda_W$ also for the $\Z$-module spanned by $W\cap \Z^n$. 
For any $\Z$-submodule $\Lambda$ of $\R^n$,
it is a fact that $\norm{\Lambda}$ (viewed as in $\wedge^{\rank(\Lambda)}\R^n$ by taking the exterior powers of a basis) is equal to the covolume (for any $\R$-linear subspace, there is a natural volume form induced from the Euclidean metric of the ambient $\R^n$) of $\Lambda$ in the $\R$-linear space spanned by it. 

Let me emphasize that the discussion so far is independent of any particular $g\in \rmG$ or $\gamma\in \rmX(\bmH,\bmP)\cap\Gamma$.
Now fix $\eta>0$ and make a choice of such $(g,\gamma)$ as in Proposition \ref{propNonDivCrit} and start the proof.
Replacing $\eta$ by a smaller one if necessary, we actually assume that
\begin{equation*}
    \norm{g p_{\gamma}(h_0)\Lambda_W}\geq \eta
\end{equation*}
for every non-trivial $p_{\gamma}(\bmH)$-stable $\Q$-subspace $W$. For the sake of notation we simply write $g$ for $g p_{\gamma}(h_0)$ and write $\rho(\xi):= p_{\gamma}( \exp(\xi))$ for $\xi \in \Xi$.

\begin{proof}
The proof proceeds by induction on the dimension of $W$ with the base case of $\dim=0$ being trivial.
So now let us assume that Equation (\ref{equaProp4.2}) has been verified for $W$ of dimension strictly less than $l$ for some $l\geq 1$ with the constant $\ep= \ep_{l-1}>0$ and take a rational subspace $W$ of dimension $l$. We need to choose $\ep_{l}>0$ independent of $g,\gamma, W$ such that Equation (\ref{equaProp4.2}) holds.

If $W$ is $\rho(\Xi)$-stable, then we are done by taking $\ep_{l}:= \min\{\eta,\ep_{l-1}\}$. Otherwise there exists $\xi_1 \in \Xi$ such that $\dim \, (\rho(\xi_1) W \cap W) < l$.

 If $\xi_1 W + W$ is $\rho(\Xi)$-stable, then we stop. Otherwise we can find $\xi_2\in \Xi$ such that $(\rho(\xi_1) W + W)\cap \rho(\xi_2) W$ has dimension $<l$. 
 Continuing this way we find $\xi_1,...,\xi_k$ in $\Xi$ such that $W+\rho(\xi_1) W+...+\rho(\xi_k) W$ is $\rho(\Xi)$-stable for the first time.
 
 For $i=1,...,k$, write $\Lambda_i$ for $\Lambda_{\rho(\xi_i)W}$ and write $\Lambda_0$ for $\Lambda_W$. For $g'\in \SL_n(\R)$ and $v\in V$, let $\norm{v}_g:=\norm{gv}$.
 Then 
 \begin{equation*}
     \rho(\xi_i)\Lambda_W \subset \frac{1}{Z_0} \cdot \Lambda_{\rho(\xi_i)W}
     \implies
      \rho(\xi_i)(Z_0\cdot \Lambda_W) \subset \Lambda_i=\Lambda_{\rho(\xi_i)W},
 \end{equation*}
 hence
 \begin{equation*}
     \norm{\rho(\xi_i) \cdot \Lambda_W }_{g'} 
     = Z_0^{-\dim W} \norm{\rho(\xi_i)(Z_0\cdot \Lambda_W)}
     \geq {Z_0^{-n}} 
     \norm{\Lambda_i}_{g'}
 \end{equation*}
 for all $g'\in \SL_n(\R)$.
 For a $\Z$-submodule $\Lambda$ of $\Z^n$ that may not be \textit{primitive} (i.e. intersection of its $\Q$-linear span with $\Z^n$ is equal to itself), let $\overline{\Lambda}$ be the unique primitive one that contains it. 
 Abbreviate $\overline{(\Lambda_0+...+\Lambda_i)}\cap \Lambda_{i+1}$ as $\Lambda'_{i+1}$. 
 By construction $\Lambda'_i$'s are all primitive and of rank strictly smaller than $l$.
 
 By $(C,\alpha)$-goodness, we find $\delta_l>0$ (depending only on $\ep_{l-1}$, $(C,\alpha)$ and $O'$) such that for all rational subspace $W'$ of dimension smaller than $l$,
 \begin{equation*}
     \Leb
     \left\{
      o\in O' \,\middle\vert\,
      \norm{g \rho(o)\Lambda_{W'}} \geq \delta_l
     \right\} 
     \geq \Leb(O') \cdot (1-\frac{1}{2n}).
 \end{equation*}
 Therefore we find $o_{l}\in O'$ such that 
 \begin{equation*}
     \norm{g\rho(o_l)\Lambda_i'} \geq \delta_l
 \end{equation*}
 for all $i$. 
 Finally, write $\rho(\xi_0)=id$ formally, then
 \begin{equation*}
     \begin{aligned}
          & \prod_{i=0,...,k} \norm{\rho(\xi_i)\Lambda_W}_{g\rho(o_l)}\\
          \geq & Z_0^{-n^2}
          \prod_{i=0,...,k} \norm{\Lambda_i}_{g\rho(o_l)}\\
          \geq & Z_0^{-n^2}
          \norm{\Lambda'_1}_{g\rho(o_l)} \cdot \norm{\Lambda_0+ \Lambda_1}_{g\rho(o_l)} \cdot  \prod_{i=2,...,k} \norm{\Lambda_i}_{g\rho(o_l)}\\
          \geq & Z_0^{-n^2}
           \norm{\Lambda'_1}_{g\rho(o_l)} \cdot 
           \norm{\overline{(\Lambda_0+ \Lambda_1)}\cap \Lambda_2}_{g\rho(o_l)} \cdot 
           \norm{\overline{\Lambda_0+\Lambda_1}+\Lambda_2}_{g\rho(o_l)} \cdot \prod_{i=3,...,k} \norm{\Lambda_i}_{g\rho(o_l)}\\
           \geq & Z_0^{-n^2}
          \prod_{i=1,...,k} \norm{\Lambda'_i}_{g\rho(o_l)} \cdot \norm{\overline{\Lambda_1+...+\Lambda_k}}_{g\rho(o_l)}\\
          =& Z_0^{-n^2}
          \prod_{i=1,...,k} 
          \norm{g\rho(o_l)\Lambda'_i} \cdot \norm{g\rho(o_l)\overline{\Lambda_1+...+\Lambda_k}}\\
          \geq & Z_0^{-n^2} \delta_l^k \eta 
     \end{aligned}
 \end{equation*}
 where we have repeatedly used Lemma \ref{lemGeometryOfNumber} below.
 So one of $\norm{g\rho(o_l)\rho(\xi_i)\Lambda_W}$ is at least ${(Z_0^{-n^2} \delta_l^k \eta )}^{1/(k+1)}$. And we are done with taking
 \begin{equation*}
     \ep_l := \min_{k} {(Z_0^{-n^2} \delta_l^k \eta )}^{1/(k+1)}.
 \end{equation*}
\end{proof}

The following well-known fact was crucial to our proof above. It is not clear to us whether it can be generalized to an arbitrary semisimple group.
\begin{lem}\label{lemGeometryOfNumber}
For any $g\in \SL_n(\R)$ and two primitive $\Z$-submodules $A$ and $B$ of $\Z^n$, we have
\begin{equation*}
    \norm{A}_g\norm{B}_g \geq \norm{A\cap B}_g \norm{A+B}_g.
\end{equation*}
\end{lem}

\section{The general case}\label{secGeneralcase}

\subsection{Proof of the Main theorem}

Let us fix an embedding $\iota: \bmG \to \bmG'$ and a maximal compact subgroup $\rmK$ of $\rmG$. By \cite{Mos552}, then there exists a maximal compact subgroup $\rmK'$ of $\rmG'$ such that $\iota_{\rmK'}\vert_{\rmG}=\iota_{\rmK}$.
Therefore $\iota$ induces an embedding $\rmX:=\rmG/\rmK \hookrightarrow \rmX' := \rmG'/\rmK'$.  
For an element $g\in\rmG$, write $[g]$ for its image in the quotient $\rmX$.

We also assume that $\Gamma$ and $\Gamma'$ are arithmetic lattices in  $\rmG_{\Q}$ and $\rmG'_{\Q}$ respectively such that $\iota(\Gamma)\subset \Gamma'$.

 The letter $\bmP$ stands for parabolic subgroups of $\bmG$ and $\bmQ$ stands for those of $\bmG'$.
And for a cocharacter $a_t: \bmG_m \to \bmG$, define 
\begin{equation*}
    \begin{aligned}
         \bmP_{a_t}:&= \left\{
          g \in \bmG \,\Big\vert\,
          \lim_{t\to \infty} a_t g a_t^{-1} \text{ exists}
         \right\},\\
         \bmQ_{a_t}:&= \left\{
          g \in \bmG' \,\Big\vert\,
          \lim_{t\to \infty} a_t g a_t^{-1} \text{ exists}
         \right\}.
    \end{aligned}
\end{equation*}

The main goal of this section is to prove
\begin{thm}\label{thmReduceTHM1toSLN}
If Theorem \ref{thmMainThmSec3} is true for $\bmG'$ then Theorem \ref{thmMainThmSec3} is true for $\bmG$.
\end{thm}
As we have proved Theorem \ref{thmMainThmSec3} for $\SL_N$, this concludes the proof.

Borel--Serre compactifications do not seem to have functorialities. Nevertheless the partial compactifications of $(\bmG,\rmK)$ and $(\bmG',\rmK')$ are somewhat related.

\begin{prop}\label{propLimGammaCell}
Let $\bmQ$ be a $\Q$-parabolic subgroup of $\bmG'$.
If there exists a sequence $(g_n)$ in $\rmG$, $(\lambda_n)$ in $\Gamma'$  such that $([\lambda_n g_n ])$ converges to some point in $e(\bmQ)$ in ${}_{\Q}\overline{\rmX'}^{\BS}$. Then after passing to a subsequence, there exists $(\gamma_n)$ in $\Gamma$ and $c \in \Gamma'$ such that $\lambda_n=c\gamma_n $ for every $n$.
\end{prop}

\begin{prop}\label{propLimCell}
Let $\bmQ$ be a $\Q$-parabolic subgroup of ${\bmG'}$. 
Assume that $(g_n)$ is a sequence  in $\rmG$ such that $([g_n])$ converges to an element of $e(\bmQ)$ in ${}_{\Q}\overline{\rmX'}^{\BS}$, then there exists a $\Q$-cocharacter $\bma_t:\bmG_m \to \bmG$ such that
\begin{enumerate}
    \item  $\bmQ=\bmQ_{\bma_t}$;
    \item $([g_n])$ converges to an element of $e(\bmP)$ 
    in ${}_{\Q}\overline{\rmX}^{\BS}$ with $\bmP:=\bmP_{\bma_t}$.
\end{enumerate}
\end{prop}

Now we comes to the proof of Theorem \ref{thmReduceTHM1toSLN} assuming the above two propositions.

\begin{proof}[Proof of Theorem \ref{thmReduceTHM1toSLN}]
So we start with a sequence $(g_n)$ of $\bmG$, a $\Q$-subgroup $\bmH$ of $\bmG$ and a nonempty bounded open subset $\calO$ of $\rmH$. By assumption, after passing to a subsequence, there exists
\begin{itemize}
    \item a $\Q$-parabolic subgroup $\bmQ$ of $\bmG'$;
    \item and a sequence $(\lambda_n)$ of $\Gamma'$
\end{itemize}
such that 
\begin{itemize}
    \item $\lambda_n^{-1}\bmH\lambda_n$ is contained in $\bmQ$ for all $n$;
\end{itemize}
and if we write 
\begin{equation*}
    \lambda_n^{-1} g_n = u_n a_n m_n k_n
\end{equation*}
for the its horospherical coordinates with respect to $(\bmQ,\rmK')$,
then
\begin{itemize}
    \item $\alpha(a_n) \to \infty$ for all $\alpha \in \Delta(\bmA_{\bmQ,\rmK'},\bmQ)$
\end{itemize}
and there exists $(o_n)$ bounded in $\calO \subset \rmH$, $(b_n)$ bounded in $\rmQ$ and $(q_n)$ in $\rmQ\cap \Gamma'$ such that 
\begin{equation*}
    \lambda_n^{-1}o_n \lambda_n u_n m_n = q_n b_n.
\end{equation*}

Consequently,
\begin{equation*}
    q_n^{-1}\lambda_n^{-1} o_ng_n = q_n^{-1}(\lambda_n^{-1} o_n \lambda_n) (u_n a_n m_n k_n)
    = b_n a_n  k_n.
\end{equation*}
Thus, after passing to a subsequence, $([q_n^{-1}\lambda_n^{-1} o_ng_n])$ converges to a point in $e(\bmQ)$. 
By Proposition \ref{propLimGammaCell}, there exists a sequence $(\gamma_n)$ in $\Gamma$ and $c\in \Gamma'$ such that $q_n^{-1}\lambda_n^{-1}=c\gamma_n^{-1}$. 
So $([\gamma_n^{-1}o_n g_n])$ converges to a point of $e(\bmQ^c)$. 
By Proposition \ref{propLimCell}, there exists a $\Q$-cocharacter $a_t : \bmG_m \to \bmG$ such that 
\begin{itemize}
    \item $\bmQ^c=\bmQ_{a_t}$;
    \item $([\gamma_n^{-1}o_n g_n])$ converges to a point of $e(\bmP)$ in ${}_{\Q}\overline{\rmX}^{\BS}$ with $\bmP:=\bmP_{a_t}$.
\end{itemize}
Now we claim that the sequence $(\gamma_n)$ and the $\Q$-parabolic subgroup $\bmP$ satisfies the conclusion of Theorem \ref{thmMainThmSec3}.

First we have 
\begin{equation*}
\begin{aligned}
         \gamma_n^{-1} \bmH \gamma_n  &= c^{-1}q_n^{-1}\lambda_n^{-1} \bmH \lambda_n q_n c\\
         &\subset c^{-1}q_n^{-1}\bmQ q_n c \cap \bmG \\
         & = c^{-1}\bmQ c \cap \bmG =\bmQ_{a_t}\cap \bmG = \bmP.
\end{aligned}
\end{equation*}

Now we write 
\begin{equation*}
    g_n':= \gamma_n^{-1} g_n = u_n' a_n' m_n' k_n'
\end{equation*}
under the horospherical coordinate with respect to $(\bmP,\rmK)$. It remains to show that 
\begin{itemize}
    \item[1.] $\alpha(a_n')\to \infty$  for all $\alpha \in \Delta(\bmA_{\bmP,\rmK},\bmP)$;
    \item[2.] there exists a sequence $(o_n')$ in $\calO$ such that 
    $[\gamma_n^{-1} o_n'\gamma_n u_n'm_n']$ remains bounded in $\Gamma \cap \rmP \bs \rmP$.
\end{itemize}
We will prove this for $o_n'=o_n$.

Write
\begin{equation*}
    g_n''=\gamma_n^{-1}o_n g_n = u_n'' a_n'' m_n'' k_n'',
\end{equation*}
in the horospherical coordinate associated with $(\bmP,\rmK)$,
then by assumption $(u_n'')$ and $(m_n''k_n'')$ are both bounded and $\alpha(a_n'')\to \infty$  for all $\alpha \in \Delta(\bmA_{\bmP,\rmK},\bmP)$.
As $\gamma_n ^{-1} o_n \gamma_n= o_n^{\gamma_n}$ is contained in $\bmP$, 
by writing
\begin{equation*}
    g_n''= (\gamma_n^{-1}o_n\gamma_n)\cdot (\gamma_n^{-1}g_n)
    =(\gamma_n^{-1}o_n\gamma_n)\cdot ( u_n' a_n' m_n' k_n'),
\end{equation*}
we see that
\begin{itemize}
    \item $a_{\bmP}(o_n^{\gamma_n})a_n' =  a_n'' $;
    \item $u_{\bmP}(o_n^{\gamma_n}){}^{a_{\bmP}(o_n^{\gamma_n})m_{\bmP}(o_n^{\gamma_n})}u_n'=u_n''$;
    \item $m_{\bmP}(o_n^{\gamma_n}) m_n' k_n'=m_n''k_n''$.
\end{itemize}
The second and the third one implies that
\begin{equation}\label{equaThm5.1}
\begin{aligned}
     \gamma_n^{-1} o_n \gamma_n u_n' m_n' k_n'
    =& u_{\bmP}(o_n^{\gamma_n}) a_{\bmP}(o_n^{\gamma_n})m_{\bmP}(o_n^{\gamma_n})u_n'm_n'k_n'\\
    =& u_n'' a_{\bmP}(o_n^{\gamma_n}) m_{\bmP}(o_n^{\gamma_n}) m_n'k_n'\\
    =&u_n'' a_{\bmP}(o_n^{\gamma_n}) m_n''k_n''.
\end{aligned}
\end{equation}

We first assume that $\left(a_{\bmP}(o_n^{\gamma_n})\right)$ is bounded. 
Then $a_n'$ is bounded away from $a_n''$, which proves the first part. And by Equation (\ref{equaThm5.1}) above, the second part also follows.

Now we prove our assumption. Note that if we had assumed that $\bmH$ has no non-trivial $\Q$-characters, then it automatically holds as $a_{\bmP}(o_n^{\gamma_n})=1$.

For each $n$, $h\mapsto \gamma_n^{-1}h\gamma_n$ induces a homomorphism from $\bmH$ to $\bmP$, which then induces a homomorphism $\overline{c}_{\gamma_n^{-1}}$ from $\bmH/{}^{\circ}\bmH \to \bmP / {}^{\circ}\bmP$. As $\bmX(\bmH,\bmP)$
is an affine variety and $\bmH/{}^{\circ}\bmH$, $\bmP / {}^{\circ}\bmP$ are both tori, we apply rigidity of diagonalizable groups \cite[Proposition 3.2.8]{Spr98} to see that, after passing to a subsequence, $\overline{c}_{\gamma_n^{-1}}=\overline{c}_{\gamma_1^{-1}}$. This implies that the image of $a_{\bmP}(o_n^{\gamma_n})$ in $\bmP / {}^{\circ}\bmP$  is bounded. But the projection map induces a proper map from $\rmA_{\bmP,\rmK}$ to $\rmP / {}^{\circ}\rmP$, hence  $\left(a_{\bmP}(o_n^{\gamma_n})\right)$ is also bounded.

\end{proof}

\subsection{Basics of geometric invariant theory}
In this section we collect some basic facts from geometric invariant theory to be used in the proof of Proposition \ref{propLimGammaCell} and \ref{propLimCell}.

Recall that $\bmG$ has been assumed to be reductive.
Let $(\rho,\bmV)$ be a $\Q$-representation of $\bmG$, we say that a vector $v\in \bmV$ is \textbf{unstable} (with respect to the $\bmG$-action) iff $\bmG\cdot v$ contains $0$ in its closure.  
We denote by $\Q[\bmV]_0$ the ideal of polynomials vanishing at $0$, $\Q[\bmV]^{\bmG}$ the algebra of invariant polynomials and $\Q[\bmV]^{\bmG}_0$  their intersection, which is an ideal of $\Q[\bmV]^{\bmG}$.

\begin{thm}\cite[Proposition 2.2]{Mumford94}
A $\Q$-vector $v$ is unstable if and only if for every $f\in \Q[\bmV]_0^{\bmG}$, $f(v)=0$.
\end{thm}

\begin{thm}\cite[Theorem 1.1]{Mumford94}
The ideal $\Q[\bmV]^{\bmG}_0$ is finitely generated. Therefore to test whether $v$ is unstable, it suffices to test $f(v)=0$ for finitely many polynomials $f$ in $\Q[\bmV]_0^{\bmG}$. 
\end{thm}

Let us fix such a finite test set $\scrA_{\bmG,\bmV}$ of generators of $\Q[\bmV]_0^{\bmG}$. The following corollary is a direct consequence of the above two lemmas.

\begin{coro}\label{coroUnstaCritEpsilon}
Fix an integral structure $\bmV(\Z)$ of $\bmV$ and a norm on $\bmV(\R)$. Then there exists $\ep_1>0$ such that $v\in \bmV(\Z)$ is unstable if and only if for each $f \in \scrA_{\bmG,\bmV}$, $|f(v)|\leq \ep_1$. Consequently there exists $\ep>0$ such that if $v\in \bmV(\Z) $ satisfies $\norm{gv}\leq \ep$ for some $g\in \bmG(\R)$ then $v$ is unstable.
\end{coro}

We shall also use the Hilbert--Mumford criterion:
\begin{thm}\cite[Theorem 2.1]{Mumford94}\label{thmHilbertMumford}
Suppose $v\in \bmV(\Q)$ is $\bmG$-unstable then there exists a $\Q$-cocharacter $a_t:\bmG_m \to \bmG$ such that $\lim_{t\to\infty} a_t v =0$.
\end{thm}

\subsection{Relations between Siegel sets}

In the proof of Proposition \ref{propLimGammaCell} and \ref{propLimCell} we will need to transfer between Siegel sets associated with different parabolic groups. Regarding this, the following three lemmas are very crucial.

\begin{lem}\label{lemGoingDown}
Assume $\bmP_1\subset \bmP_2 $ is a pair of $\Q$-parabolic subgroups of $\bmG$. Then any Siegel set associated with $(\bmP_2,\rmK)$ is contained in a Siegel set associated with $(\bmP_1,\rmK)$. In particular, part (2) of Theorem \ref{thmReducThy} is also satisfied for every Siegel set associated with possibly non-minimal $\Q$-parabolic subgroups.
\end{lem}

This lemma is a simple consequence of Section \ref{secRelationHoro1} on comparing different horospherical coordinates. 
The converse is also not far away from being true. To make a statement, let us fix a $\Q$-parabolic subgroup $\bmP$ of $\bmG$ and a sequence $(g_n)$ of $\rmG$ that belongs to a fixed Siegel set associated with $(\bmP,\rmK)$. After passing to a subsequence, we assume that for each $\alpha \in \Delta(\bmS_{\bmP},\bmP)$, either one of the following two happens
\begin{enumerate}
    \item $(\alpha(a_{\bmP}(g_n))) \to +\infty$;
    \item $(\alpha(a_{\bmP}(g_n)))$ remains bounded.
\end{enumerate}
Let $I \subset \Delta(\bmS_{\bmP},\bmP)$ be those $\alpha$ for which (2) holds.

\begin{lem}\label{lemGoingUp}
The sequence $([g_n])$ converges to a point of $e(\bmP_I)$ in ${}_{\Q}\overline{\rmX}^{\BS}$.
\end{lem}

\begin{proof}
By Section \ref{secRelationHoro1}, it suffices to show that when we decompose $a_{\bmP}(g_n) =a_{\bmP_I}(g_n) \cdot a_{\bmP}^I(g_n) $, then 
\begin{enumerate}
    \item for every $\alpha \notin I$, $\alpha(a_{\bmP_I}(g_n)) \to \infty$;
    \item $\{a_{\bmP}^I(g_n)\}$ is bounded.
\end{enumerate}
Indeed, for $\alpha \in I$, $\alpha(a_{\bmP}(g_n))=\alpha(a_{\bmP}^I(g_n))$ is bounded (both away from $0$ and $+\infty$). Therefore (2) is true. 
Then for $\alpha \notin I$, $\alpha(a_{\bmP_I}(g_n))$ is of bounded distance away from $\alpha(a_{\bmP}(g_n))$, which tends to $\infty$. So we are done.
\end{proof}

These two lemmas allows us to go up and down between different Siegel sets within the same ambient group. In general it is not clear how we can move between parabolic groups $\bmP$ and $\bmQ$ associated respectively with $\bmG$ and $\bmG'$ even when $\bmQ \cap \bmG=\bmP$. However, there is one case where this is possible.

Assume that  $a_t:\bmG_m \to \bmG$ is a cocharacter over $\Q$, then we have
\begin{lem}\label{lemGoingBtw}
For an element $g$ of $\rmG$, its horospherical coordinate associated with $(\bmP_{a_t},\rmK)$ coincides with the one associated with $(\bmQ_{a_t},\rmK')$.
\end{lem}

\begin{proof}
For simplicity write $\bmP:=\bmP_{a_t}$ and $\bmQ=\bmQ_{a_t}$.
Indeed, it suffices to check that $\iota$ sends $\bmU_{\bmP}$ to $\bmU_{\bmQ}$, $\bmA_{\bmP,\rmK}$ to $\bmA_{\bmQ,\rmK'}$ and $\bmM_{\bmP,\rmK}$ to $\bmM_{\bmQ,\rmK'}$.

Indeed $\iota(\bmU_{\bmP})=\bmU_{\bmQ}$ follows from definitions of $\bmP_{a_t}$ and $\bmQ_{a_t}$.

The cocharacter $a_t$ is conjugate within $\bmP$ to a unique $\R$-cocharacter $b_t$ of $\bmA_{\bmP,\rmK}$. We still have $\bmP=\bmP_{b_t}$ and $\bmQ=\bmQ_{b_t}$. Moreover, as the image of $b_t$ is $\iota_{\rmK}$-stable, we deduce that its centralizer in $\bmP$ (resp. $\bmQ$) is $\iota_{\rmK}$ (resp. $\iota_{\rmK'}$)-stable and hence by uniqueness is equal to $\bmL_{\bmP,\rmK}$ (resp. $\bmL_{\bmQ,\rmK'}$). Thus $\iota(\bmL_{\bmP,\rmK'})$ is contained in $\bmL_{\bmQ,\rmK'}$ and we are done.
\end{proof}

As the reader might notice, the above proof works as long as $\bmU_{\bmP}$ is contained in $\bmU_{\bmQ}$. But the situation stated above is where we are going to apply the lemma.

\subsection{Proof of Proposition \ref{propLimGammaCell}}

\usetikzlibrary{decorations.pathmorphing}

The diagram is a brief summary of the proof.
\[
\begin{tikzcd}
      & \bmQ'\arrow[d,squiggly] & \bmQ \arrow[l,squiggly,"\text{conjugate}"']\\
      & \bmQ_{b_t} \arrow[d,squiggly] & \\
      \bmP_0=\bmP_{a_t}\arrow[r,squiggly] & \bmQ_{a_t}
\end{tikzcd}
\]

\begin{proof}

Now we start the formal proof.
First we fix  a norm on $\Lie(\bmG')_{\R}$ and an integral structure on $\Lie(\bmG')$ that is preserved by $\Gamma'$ under the Adjoint action. Let $k$ be the dimension of $\bmU_{\bmQ}$, then we can find a basis $v_1,...,v_k$ of $\Lie(\bmU_{\bmQ})$ consisting of integral vectors.

By assumption $(\lambda_ng_n)$ converges to a point in $e(\bmQ)$, thus $\alpha(a_{\bmQ}(\lambda_ng_n))^{-1}$ converges to $0$ for all $\alpha\in \Phi(\bmA_{\bmQ,\rmK},\bmQ)$. This implies that 
\begin{equation*}
\begin{aligned}
      &g_n^{-1} \lambda_n^{-1} v_i \to 0 ,\;\text{ for }i=1,...,k
    \implies &
     g_n^{-1}\lambda_n^{-1} \oplus_{i=1,...,k} v_i \to 0.
\end{aligned}
\end{equation*}

According to Corollary \ref{coroUnstaCritEpsilon}, applied to the direct product of $k$-copies of Adjoint representation restricted to (the diagonal action of) $\bmG$ with the product norm and integral structure, we find an $\ep>0$ such that if there is a vector of norm less than $\ep$ in the $\rmG$-orbit of an integral vector $w$, then $w$ is $\bmG$-unstable.

Now we pick some $n_0$ such that 
\begin{equation*}
    \norm{
    g_{n_0}^{-1} \lambda_{n_0}^{-1} \oplus_{i=1,...,k}v_i
    } \leq \ep.
\end{equation*}
Thus $ \lambda_{n_0}^{-1}\oplus_{i=1,...,k} v_i$ is $\bmG$-unstable. By Theorem \ref{thmHilbertMumford}, we find a $\Q$-cocharacter $b_t : \bmG_m \to \bmG$ such that 
\begin{equation*}
    \lim_{t\to\infty} b_t \lambda_{n_0}^{-1} v_i = 0,\;
    \text{ for }i=1,...,k.
\end{equation*}

Therefore, it follows that the unipotent radical of $\bmQ_{b_t}$ contains $\lambda_{n_0}^{-1}\bmU_{\bmQ}\lambda_{n_0}$, or equivalently, $\bmQ_{b_t}$ is contained in $\bmQ':=\lambda_{n_0}^{-1}{\bmQ}\lambda_{n_0}$.

Fix a maximal $\Q$-split torus $\bmS$ in $\bmP_{b_t}$ such that $b_t$ is contained in $\bmX_*(\bmS)$. Then for any $a_t$ sufficiently close to $b_t$ in $\bmX_*(\bmS)$, one has that $\bmQ_{a_t}$ is contained in $\bmQ_{b_t}$. Now we choose such an $a_t$ that is generic, i.e. the associated $\bmP_0:=\bmP_{a_t}$ is a minimal $\Q$-parabolic subgroup of $\bmG$.

By Theorem \ref{thmReducThy}, after passing to a subsequence, there exists $c_0\in \bmG(\Q)$ and a Siegel set $\frakS_{\omega_0,t}$ associated with $(\bmP_0,\rmK)$ such that $\{g_n\}$ is contained in $(\Gamma\cap\rmG)c_0\frakS_{\omega_0,t}$.
Hence there exists $\gamma_n\in \Gamma\cap \rmG$ such that $g_n$
is contained in $\gamma_n^{-1} c_0 \frakS_{\omega_0,t}$. 
By Lemma \ref{lemGoingBtw}, $c_0^{-1}\gamma_ng_n$ is also contained in some Siegel set associated with $(\bmQ_{a_t},\rmK')$.

On the other hand, as $(\lambda_n g_n)$ converges to a point in $e(\bmQ)$, $(\lambda_{n_0}^{-1}\lambda_n g_n)$ converges to a point in $e(\bmQ')$. In particular $(\lambda_{n_0}^{-1}\lambda_n g_n)$ is contained in a Siegel set associated with $(\bmQ',\rmK')$. 
By Lemma \ref{lemGoingDown}, this sequence is also contained in a Siegel set associated with $(\bmQ_{a_t},\rmK')$. By enlarging the Siegel set, if possible, there is a common Siegel set $\frakS$ associated with $(\bmQ_{a_t},\rmK')$ such that 
\begin{equation*}
    g_n \in \gamma_n^{-1}c_0\frakS \bigcap 
    \lambda_n^{-1}\lambda_{n_0} \frakS
\end{equation*}
for all $n$. In particular
\begin{equation*}
    \lambda_{n_0}^{-1}\lambda_n\gamma_n^{-1}c_0\frakS \bigcap 
     \frakS \neq \emptyset.
\end{equation*}
Now Theorem \ref{thmReducThy} together with Lemma \ref{lemGoingDown} implies that after passing to a subsequence,
\begin{equation*}
    \lambda_{n_0}^{-1}\lambda_n\gamma_n^{-1} =
    \lambda_{n_0}^{-1}\lambda_{n_1}\gamma_{n_1}^{-1}
\end{equation*}
for some $n_1$.
By taking $c:= \lambda_{n_1}\gamma_{n_1}^{-1}$ we find that
\begin{equation*}
    \lambda_n = c\gamma_n
\end{equation*}
for all $n$.
\end{proof}

\subsection{Proof of Proposition \ref{propLimCell}}

The following two diagrams form a brief summary of the proof.
\[
\begin{tikzcd}
        & \bmQ \arrow[d, squiggly]\\
      \bmP_{b_t} \arrow[r, squiggly] & \bmQ_{b_t}
\end{tikzcd}
\quad\quad
\begin{tikzcd}
      \bmP_{a_t}  & \bmQ= (\bmQ_{b_t})_I =\bmQ_{a_t} \arrow[l, squiggly]\\
      & \bmQ_{b_t}\arrow[u, squiggly]
\end{tikzcd}
\]

\begin{proof}

Now let us start the formal proof.
As in the last section (specialized to the situation $\lambda_n\equiv 1$), we find a $\Q$-cocharacter $b_t: \bmG_m \to \bmG$ such that  $\bmQ_{b_t}$ is contained in $\bmQ$. 

Let 
\begin{equation*}
    g_n = u_n a_n m_nk_n
\end{equation*}
be the horospherical coordinate of $g_n$ with respect to $(\bmP_{b_t},\rmK)$. 
By Lemma \ref{lemGoingBtw}, this is also the horospherical coordinate of $g_n$ with respect to $(\bmQ_{b_t},\rmK')$.
As $([g_n])$ converges to a point of $e(\bmQ)$, $\{g_n\}$ is contained in a Siegel set associated with $(\bmQ, \rmK')$. 
By Lemma \ref{lemGoingDown}, $\{g_n\}$ is then contained in some Siegel set associated with $(\bmQ_{b_t},\rmK')$. 
By Lemma \ref{lemGoingUp}, after passing to a subsequence, there exists $I\subset \Delta(\bmA_{\bmQ_{b_t}},\bmQ_{b_t})$ such that 
for $\alpha \in \Delta(\bmA_{\bmQ_{b_t}},\bmQ_{b_t})$,
\begin{equation*}
    \begin{aligned}
         \alpha(a_n) \to \infty &\iff \alpha \in I\\
         \alpha(a_n) \text{ bounded } &\iff \alpha \notin I
    \end{aligned}
\end{equation*}
and $\bmQ= (\bmQ_{b_t})_I$. 
Here bounded means that, as $\{g_n\}$ is contained in a Siegel set, both from above and from below.

Take a $\Q$-lift of $\widetilde{\bmS}_{\bmP_{b_t}}$ of ${\bmS}_{\bmP_{b_t}}$ in $\bmP_{b_t}$. Then $\widetilde{\bmS}_{\bmP_{b_t}}$ is canonically conjugate to $\bmA_{\bmP_{b_t}}$ via some element in $\rmU_{b_t}$ and we let $s_n$ be the image of $a_n$ under this conjugation.

For $\alpha \in  \Delta(\bmA_{\bmQ_{b_t}},\bmQ_{b_t})$, write $\widetilde{\alpha}$ for its pull-back to  $\widetilde{\bmS}_{\bmP_{b_t}}$, which are actually $\Q$-characters of  $\widetilde{\bmS}_{\bmP_{b_t}}$. So we have

\begin{equation*}
    \begin{aligned}
          \widetilde{\alpha}(s_n) \to \infty &\iff \alpha \in I\\
          \widetilde{\alpha}(s_n) \text{ bounded } &\iff \alpha \notin I.
    \end{aligned}
\end{equation*} 
Therefore, if we denote by $\la \cdot,\cdot \ra$ the natural pairing between $\bmX^*(\widetilde{\bmS}_{\bmP_{b_t}})$ and $\bmX_*(\widetilde{\bmS}_{\bmP_{b_t}})$, then there exists $a_t \in \bmX_*(\widetilde{\bmS}_{\bmP_{b_t}})$ such that 
\begin{equation*}
    \begin{aligned}
         \la \widetilde{\alpha}, a_t \ra >0  &\iff \alpha \in I\\
         \la  \widetilde{\alpha}, a_t \ra =0 &\iff \alpha \notin I.
    \end{aligned}
\end{equation*}
This shows that $\bmQ=(\bmQ_{b_t})_{I}=\bmQ_{a_t}$. By applying Lemma \ref{lemGoingBtw}, we see that $([g_n])$ converges to a point of $e(\bmP_{a_t})$ in ${}_{\Q}\overline{\rmX}^{\BS}$.
\end{proof}

\section*{Acknowledgement}
We are grateful to discussions with Osama Khalil, Nimish Shah and Pengyu Yang.

\bibliographystyle{amsalpha}
\bibliography{ref}
\end{document}